\input ./style/arxiv-vmsta.cfg
\documentclass[numbers,compress,v1.0.1]{vmsta}
\usepackage{eucal,mathrsfs}
\usepackage{mathbh}

\newtheorem{thm}{Theorem}
\newtheorem{lemma}{Lemma}
\newtheorem{cor}{Corollary}
\newtheorem{proposition}[thm]{Proposition}

\theoremstyle{definition}
\newtheorem{remark}{Remark}
\newtheorem{example}{Example}

\newcommand{\real}{\mathbb{R}}
\newcommand{\rn}{{\mathbb{R}^n}}
\newcommand{\Ee}{\mathbb{E}}
\newcommand{\I}{\mathbh{1}}
\newcommand{\Pp}{\mathbb{P}}
\newcommand{\Sss}{\mathbb{S}}
\newcommand{\sq}{\mbox{\tiny$\square$} }
\newcommand{\Ff}{\mathcal{F}}

\volume{2}
\issue{2}
\pubyear{2015}
\firstpage{107}
\lastpage{129}
\doi{10.15559/15-VMSTA26}% Updated by VTEXPTS2LaTeX.exe, 04.06.2015
\startlocaldefs
\allowdisplaybreaks
\endlocaldefs

\begin{document}
\begin{frontmatter}

\title{On the Feynman--Kac semigroup for some Markov processes}
\author{\inits{V.}\fnm{Victoria}\snm{Knopova}}\email
{vicknopova@googlemail.com}
%\cortext[cor1]{Corresponding author.}

%\fnref{f1}
%\fntext[]{Some remarks}

\address{V.M.\ Glushkov Institute of Cybernetics, NAS of Ukraine,\\
40, Acad.\ Glushkov Ave., 03187, Kiev, Ukraine}

\markboth{V. Knopova}{On the Feynman--Kac semigroup for some Markov processes}

\begin{abstract}
For a (non-symmetric) strong Markov process $X$, consider the
Feynman--Kac semigroup
\[
T_t^A f(x):= \Ee^x \bigl[ e^{A_t }
f(X_t) \bigr],\quad x\in\rn,\ t>0,
\]
where $A$ is a continuous additive functional of $X$ associated with
some signed measure. Under the assumption that $X$ admits a transition
probability density that possesses upper and lower bounds of certain
type, we show that the kernel corresponding to $T_t^A$ possesses the
density $p_t^A(x,y)$ with respect to the Lebesgue measure and construct
upper and lower bounds for $p_t^A(x,y)$.
Some examples are provided.
\end{abstract}


\received{1 October 2014}% Updated by VTEXPTS2LaTeX.exe, 04.06.2015
%09:19
%
\revised{10 May 2015}% Updated by VTEXPTS2LaTeX.exe, 04.06.2015 09:19
\accepted{26 May 2015}% Updated by VTEXPTS2LaTeX.exe, 04.06.2015 09:19
\publishedonline{18 June 2015}
\end{frontmatter}

\numberwithin{equation}{section}

%s1 ###
\section{Introduction}

Let $(X_t)_{t\geq0}$ be a Markov process with the state space $\rn$.
For a Borel measurable function $V:\rn\to\real$, we can define the
functional $A_t$ of $X$ by
%
%e1 ###
\begin{equation}
\label{A10} A_t:=\int_0^t
V(X_s)ds, \quad t>0.\vspace{15pt}\vadjust{\eject}
\end{equation}
Suppose that
$\lim_{t\to0} \sup_x \Ee^x |A_t|=0$.
Then, by the Khasminski lemma %(cf. \cite[Lemma~3.3.7]{CZ95}),
there exist constants $C,b>0$, such that
%
%e2 ###
\begin{equation}
\label{gaug1} \sup_x \Ee^x e^{|A_t|}\leq
C e^{bt};
\end{equation}
see, for example, \cite[Lemma~3]{CR88} or \cite[Lemma~3.3.7]{CZ95}.
%This result is known as a gauge lemma, see \cite[Lemma~3]{CR88}, also
%\cite[Lemma~1.1]{Po82}.
Estimate \eqref{gaug1} allows us to define the operator
%
%e3 ###
\begin{equation}
\label{ta1} T_t^A f(x):= \Ee^x \bigl[
e^{A_t } f(X_t) \bigr],\quad x\in\rn,\ t>0,
\end{equation}
where the function $f$ is bounded and Borel measurable. The family of
operators $(T_t^A)_{t\geq0}$ forms a semigroup, called the \emph
{Feynman--Kac semigroup}.

Feynman--Kac semigroup is well studied in the case of a Brownian
motion (see \cite{Sz98,CD00,CZ95,BM90}); in
particular, in \cite{BM90} more general functionals are treated.
The case of a general Markov process is much more complicated; see,
however, \cite[Chap.~3.3.2]{CZ95} and \cite{CD00}. The essential
condition on the process, stated in the papers cited, is that the
Markov process $X$ is symmetric and possesses a transition probability
density $p_t(x,y)$.

In this paper, we construct and investigate the Feynman--Kac
semigroups for a wider class of Markov processes. First, we construct
the Feynman--Kac semigroup for a (non-symmetric) Markov process,
admitting a transition density. We also treat a more general class of
functionals $A_t$, that is, in our setting the functional $A_t$ is not
necessarily of the form \eqref{A10}, but is constructed by means of
some measure $\varpi$, which is in the \emph{Kato class} with respect
to the transition probability density of $X$ (cf. \eqref{K1}). The approach used in
\cite
{CB11} allows us to show the existence of the kernel $p_t^A(x,y)$ of
the semigroup $(T_t^A)_{t\geq0}$ and to give its representation. The
method from \cite{CB11} relies on the construction of the \emph{Markov
bridge density}, which in turn employs the regularity properties of the
transition probability density of the initial process $X$ rather than its symmetry.

In such a way, this prepares the base for the main result of the paper,
which is devoted to the investigation of the Feynman--Kac semigroup for
the particular class of processes constructed in \cite{KK13a}. In
\cite{KK12a,KK15}, we develop the approach that allows us to relate
to a pseudo-differential operator of certain type a Markov process
possessing a transition probability density $p_t(x,y)$ and construct
for this density two-sided estimates. In particular, such estimates
provide an easily checkable condition when a measure $\varpi$ belongs
to the Kato class with respect to $p_t(x,y)$. This allows us to
describe the respective continuous additive functional $A_t$ and to
show \eqref{gaug1}. Starting with the class of processes investigated
in \cite{KK13a}, we construct (see Theorem~\ref{T-FK1}) the upper and
lower estimates for the Feynman--Kac density $p^A_t(x,y)$. In
particular, we show that the structure of such estimates is
``inherited'' from the structure of the estimates on $p_t(x,y)$.
In some cases when the upper bound on $p_t(x,y)$ can be written in a
rather compact way, we can describe explicitly the Kato class of
measures. For example, this is the case if  $p_t(x,y)$ is comparable for small $t$ with the density of a symmetric stable
process; see also \cite[Cor.~12]{BJ07} for refined results. In
Proposition~\ref{pr1} we show that if the initial transition probability density
possesses an upper bound of a rather simple (polynomial) form, this
form is inherited by the Feynman--Kac density $p_t^A(x,y)$.

Up to the author's knowledge, in general, the results on two-sided
estimates of $p^A_t(x,y)$ are yet unavailable.
For $X$ being an $\alpha$-stable-like process, the estimates of the
kernel $p^A_t(x,y)$ are obtained in \cite{So06}; see also \cite{CKS13}
and the references therein for more recent results in this direction,
including two-sided estimates on $p^A_t(x,y)$ in the case when the
functional $A$ is not necessarily continuous. The approach used in
\cite
{So06,CKS13} to construct the Feynman--Kac semigroup is based
on the Dirichlet form technique. See also \cite{BS13} for yet another
approach to investigate Feynman--Kac semigroups.

The paper is organized as follows. In Section~\ref{main}, we give the
basic notions and introduce the main results. Proofs are given in
Sections~\ref{Proo1} and \ref{Proo2}. In Section~\ref{exam}, we
illustrate our results with examples.

\subsection*{Notation} For functions $f$, $g$,  by $f\asymp g$
we mean that there exist some constants $c_1, c_2>0$ such that $c_1 f(x)
\leq g(x) \leq c_2 f(x)$ for all $x\in\rn$. By $x\cdot y$ and
$\|x\|$ we denote, respectively, the scalar product and the norm in
$\rn$, and $\Sss^n$ denotes the unit sphere in $\rn$. By $B_b(\rn)$ we
denote the family of bounded Borel functions on~$\rn$. By $C^k_\infty
(\rn)$ we denote the space of $k$-times differentiable functions, with
derivatives vanishing at infinity. By $c_i$,
$c$ and  $C$ we denote arbitrary positive constants. The symbols
$*$, $\sq$, and $\lozenge$
denote, respectively, the convolutions
\begin{align*}
\label{dia} (f* g) (x,y)&:= \int_\rn f(x-z)g(z-y)dz, \quad\\
(f\,\sq\, g ) (x,y)&:= \int_\rn f(x-z) g(z-y)\varpi(dz),
\end{align*}
and
\begin{equation*}
(f\,\lozenge\, g )_t(x,y):= \int_0^t
\int_\rn f_{t-s}(x,z) g_s(z,y)
\varpi(dz) ds,
\end{equation*}
where $\varpi$ is a (signed) measure.

%s2 ###
\section{Settings and the main results}\label{main}

Let $X$ be a Markov process with the state space $\rn$. We call $X$ a
\emph{Feller process} if the corresponding operator
%
%e4 ###
\begin{equation}
\label{Tt} T_t f(x):= \Ee^x f(X_t)
\end{equation}
maps the space $C_\infty(\rn)$ of continuous functions vanishing at
infinity into itself.
Assume that $X$ possesses a transition probability density $p_t(x,y)$ which satisfies
the following assumption.

\begin{itemize}

\item[\textbf{P1.}] For fixed $x\in\rn$, the mapping $y\mapsto
p_s(x,y)$ is continuous for all $s\in(0,t]$, and the mapping
$s\mapsto p_s(x,y)$ is continuous for all $x,y\in\rn$.
\end{itemize}

Recall some notions on the Kato class of measures and related
continuous additive functionals.

We say that a functional $\varphi_t$ of a Markov process $X_t$ is a
$W$-functional (see \cite[\S6.11]{Dy65}) if\vadjust{\goodbreak} $\varphi_t$ is a positive
continuous additive functional, almost surely homogeneous, and such
that $ \sup_x\Ee^x \varphi_t<\infty$. By additivity we mean that
$\varphi_t$ satisfies the following equality:
%
%e5 ###
\begin{equation}
\label{ad1} \varphi_{t+s} = \varphi_t +
\varphi_s\circ\theta_t,
\end{equation}
where $\theta_t $ is the shift operator, that is, $X_s \circ\theta_t
=X_{t+s}$. The function $v_t(x):= \Ee^x \varphi_t$ is called the
characteristic of $\varphi_t$ and determines $\varphi_t$ in the unique
way; see \cite[Thm.~6.3]{Dy65}.

A positive Borel measure $\varpi$ is said to belong to the Kato class
$S_K$ with respect to $p_t(x,y)$ if
%
%e6 ###
\begin{equation}
\lim_{t\to0} \sup_{x\in\rn} \int
_0^t \int_\rn
p_s(x,y)\varpi (dy)ds=0. \label{K1}
\end{equation}
By \cite[Thm.~6.6]{Dy65}, the condition $\varpi\in S_K$ implies that
the function
%
%e7 ###
\begin{equation}
\chi_t(x):= \int_0^t \int
_{\rn} p_s(x,y)\varpi(dy)ds \label{chit}
\end{equation}
for which the mapping $x\mapsto\chi_t(x)$ is measurable for all
$t\geq
0$, is the characteristic of some $W$-functional $\varphi_t$.

Let $\varpi=\varpi^+-\varpi^-$ be a signed measure such that $\varpi
^\pm
\in S_K$ with respect to $p_t(x,y)$. Then
%
%e8 ###
\begin{equation}
\label{chipm} \chi_t^\pm:= \int_0^t
\int_\rn p_s(x,y)\varpi^\pm(dy)ds
\end{equation}
are the characteristics of some $W$-functionals $A_t^\pm$,
respectively, that is, there exist $A_t^\pm$ such that $\chi_t^\pm(x)=
\Ee^x A_t^\pm$. Since for such functionals we have
\[
\lim_{t\to0}\sup_x \Ee^x
A_t^\pm=0,
\]
then estimate \eqref{gaug1} holds true, and thus the Feynman--Kac semigroup
$(T_t^A)_{t\geq0}$ for $A_t:=A_t^+- A_t^-$ is correctly defined.

To show that the semigroup $(T_t^A)_{t\geq0}$ can be written as
\[
T_t^A f(x)= \int_\rn f(y)
p^A_t(x,y)dy, \quad f\in B_b\bigl(\rn\bigr),
\]
and to find the representation of the density $p^A_t(x,y)$ in terms of
the probability density of the initial process, recall some notions on
Markov bridge measures.

Denote by $(\Ff_t)_{t\geq0}$ the admissible filtration related to $X$.
A \emph{Markov bridge} $X_t^{x,y}$ of $X_t$ is a Markov processes
conditioned by $X_0=x$ and $X_t=y$.
In the proof of \cite[Thm.~1]{CB11}, it is shown that under \textrm{P1}
there exists the corresponding Markov bridge measure $\Pp
_{x,y}^t$ on $\Ff_{t-}$ for $(t,x,y)$ such that $p_t(x,y)>0$.
%$$
%\Pp_{x,y}^t (X_0=x)=1, \quad\Pp_{x,y}^t (X_{t-}=X_t=y)=1.
%$$
%and $\Pp_{x,y}^t$ is absolutely continuous with respect to the law
%$P_s^x$ of $X_s^x$, i.e.
% \begin{equation}\label{p10}
% \frac{d\Pp_{x,y}^t |_{\mathcal{F}_s}}{d \Pp_x |_{\mathcal{F}_s}}=
%\frac{p_{t-s}(X_s,y)}{p_t(x,y)}.
% \Pp_{x,y}^t(dz) = \frac{p_{t-s}(z,y)}{p_t(x,y)}\Pp_s^x(dz)=
%\frac{p_{t-s}(z,y)p_s(x,z)}{p_t(x,y)} dz.
%\end{equation}
We denote by $\Ee_{x,y}^t$ the expectation with respect to $\Pp_{x,y}^t$.

The next proposition is essentially contained in \cite[Thm.~1]{CB11},
but we reformulate the result in the way convenient for our purposes.

\begin{proposition}\label{p-kato1}
Let $X$ be a Feller process, admitting the transition probability density
$p_t(x,y)$, for which assumption \emph{P1} holds. Let $\varpi=\varpi
^+-\varpi^-$ be a signed Borel measure, $\varpi^\pm\in S_K$, and
$A_t=A_t^+-A_t^-$, where $A^\pm$ are continuous additive functionals
with characteristics \eqref{chipm}, respectively. Then
\[
T^A_tf(x)= \int_{\{y:\, p_t(x,y)>0\}} f(y)
p_t^A(x,y)dy \quad\text{for any $f\in B_b\bigl(
\rn\bigr)$},
\]
where
%
%e9 ###
\begin{equation}
\label{pv10} p^A_t(x,y)=p_t(x,y)
\Ee_{x,y}^t e^{A_t}, \quad\text{$x,y \in\rn$,
$t>0$.}
\end{equation}
\end{proposition}

\begin{remark}
When $X$ is a Brownian motion, the statement of Proposition~\ref
{p-kato1} is known, see \cite{Sz98} and also \cite{BM90}. The
construction from \cite{BM90,Sz98} can be extended to the case
of a \emph{symmetric} Markov process, see \cite{CD00}. On the contrary,
the construction presented in \cite{CB11} relies on \textrm{P1} and
does not require the symmetry of the initial process.
\end{remark}

Proposition~\ref{p-kato1} implicitly gives the representation of the
function $p^A_t(x,y)$. However, when one wants to get quantitative
information about $p_t^A(x,y)$, like the upper bound on $p_t^A(x,y)$,
estimation of the expectation $\Ee_{x,y}^t e^{A_t}$ in \eqref{pv10}
appears to be non-trivial. Instead, for some class of Feller processes,
we can use another approach, which enables us to get explicitly an
upper estimate of $p_t^A(x,y)$. Namely, in \cite{KK13a} we formulated
the assumptions under which %, starting with an integro-differential
%operator,
one can construct a Feller process possessing the transition probability density
$p_t(x,y)$ satisfying assumption \textrm{P1} and admitting upper and
lower bounds of certain form. In order to make the presentation
self-contained, we quote this result below.

Let
%
%e10 ###
\begin{equation}
\label{lxd0} \mathcal{L}f(x):=a(x) \cdot\nabla f(x)+ \int_\rn
\bigl(f(x+u)-f(x)-u \cdot\nabla f(x) \I_{\{\|u\|\leq1\}} \bigr)m(x,u)\mu(du),
\end{equation}
where $f\in C_\infty^2(\rn)$, and $\mu$ is a L\'evy measure, that is,
a Borel measure such that
\[
\int_\rn\bigl(\|u\|^2\wedge1\bigr)\mu(du)<
\infty.
\]
Assume that $\mu$ satisfies the following assumption.

\begin{itemize}
\item[\textbf{A1.}] There exists $\beta>1$ such that
\[
\sup_{\ell\in\Sss^n}q^U(r \ell)\leq\beta\inf
_{\ell\in\Sss^n} q^L(r \ell)\quad\text{ for all $r>0$ large
enough, }
\]
\end{itemize}
where
%
%e11 ###
\begin{equation}
q^U(\xi):=\int_\rn\bigl[(\xi\cdot
u)^2\wedge1\bigr]\mu(du),\qquad q^L(\xi ):=\int
_{|u\cdot\xi|\leq1} (\xi\cdot u)^2\mu(du).\label{psipm}
\end{equation}
%
%\vica This condition holds true, for example, for a symmetric $
%\alpha$-stable L\'evy measure with $\alpha:= 2/\beta$; see
%\cite{KK12a} for details. %We consider more examples in Section~
%\ref{exam}.

Assume that the functions $a(x)$ and $m(x,u)$ in \eqref{lxd0} satisfy
the  assumptions \textrm{A2}--\textrm{A4} given below.

\begin{itemize}

\item[\textbf{A2.}] The functions $m(x,u)$ and $a(x)$ are measurable,
and satisfy with some constants $b_1,\, b_2,\, b_3 >0$, the inequalities
\[
b_1\leq m(x,u)\leq b_2,\qquad\big|a(x)\big|\leq
b_3, \quad x,u\in\rn.
\]

\item[\textbf{A3.}] There exist constants $\gamma\in(0,1]$ and
$b_4>0$ such that
\begin{align}
\label{M2b} \big|m(x,u)-m(y,u)\big| + \big\|a(x)-a(y)\big\| \leq b_4\bigl( \|x-y
\|^\gamma\wedge 1\bigr),\quad u,\, x,\,y \in\rn.
\end{align}

\item[\textbf{A4.}] In the case $\beta>2$, we assume that $a(x)=0$ and
the kernel $m(x,u)\mu(du)$ is symmetric with respect to $u$ for all
$x\in\rn$.
\end{itemize}

Denote by $f_{\mathrm{low}}$ and $f_{\mathrm{up}}$ the functions of the form
%
%e12 ###
\begin{equation}
\label{flu1} f_{\mathrm{low}}(x):= a_1 \bigl(1-a_2\|x
\|\bigr)_+, \qquad f_{\mathrm{up}}(x):=a_3 e^{-a_4 \|x\|}, \quad x\in
\rn,
\end{equation}
where $a_i>0$, $1\leq i\leq4$, are some constants.

Finally, define $q^* (r) :=\sup_{\ell\in\Sss^n}q^U(r \ell)$, $ r>0$.
It was shown in \cite{K13} (see also~\cite{KK12a}) that condition
\textrm{A1} implies that
\[
q^*(r)\geq r^{2/\beta}, \quad r\geq1.
\]
Note also that the continuity of $q^U$ in $\xi$ implies the continuity
of $q^*$ in $r$. Therefore, we can define its generalized inverse
%
%e13 ###
\begin{equation}
\label{rho1} \rho_t:= \inf\bigl\{ r: q^*(r) =1/t\bigr\}, \quad
t\in(0,1].
\end{equation}

\begin{thm}[\cite{KK13a}]\label{t-main}
Under assumptions \emph{A1}--\emph{A4}, the operator $(\mathcal{L},
C_\infty^2(\rn)$ extends to the generator of a Feller process, admitting
a transition probability  density $p_t(x,y)$. This density is continuous in
$(t,x,y)\in(0,\infty)\times\rn\times\rn$, and
there exist constants $a_i>0$, $1\leq i\leq4$, and a family of
sub-probability measures $\{ Q_t, t\geq0\}$ such that
%
%e14 ###
\begin{equation}
\label{upper-es} \rho_t^n f_{\mathrm{low}}\bigl((x-y)
\rho_t\bigr) \leq p_t(x,y)\leq\rho_t^n
\bigl(f_{\mathrm{up}}(\rho_t \cdot)* Q_t \bigr) (x-y),
\quad t\in(0,1], \ x,y\in \rn,
\end{equation}
where $f_{\mathrm{low}}$ and $f_{\mathrm{up}}$ are functions of the form \eqref{flu1}
with constants $a_i$, and $\rho_t$ is defined in \eqref{rho1}.
\end{thm}

The constructed process is a \emph{L\'evy type} process. In the
``constant coefficient case,'' that is, where $a(x)\equiv \mathrm{const}$ and
$m(x,u)=\mathrm{const}$, \eqref{lxd0} is just the representation of the
generator of a L\'evy process; in other words, a L\'evy type process is
the process with ``locally independent increments.'' It is known
(cf.~the Courr\`ege--Waldenfels theorem, see \cite[Thm.~4.5.21]{Ja01})
that if the class $C_c^\infty(\rn)$ of infinitely differentiable
compactly supported functions belongs to the domain $D(A)$ of the
generator $A$ of a Feller process, then on this set~$C_c^\infty(\rn)$
the operator $A$ coincides with $\mathcal{L}+{} $``Gaussian
component.'' Thus, the class of processes satisfying the conditions of
Theorem~\ref{t-main} is rather wide.

Let us show that, under the conditions of Theorem~\ref{t-main}, we have
\[
p_t(x,y)>0\quad\text{for all $ t>0$,\,\, $x,y\in\rn$.}
\]
We find the minimal $N$ such that the distance from $x$ to $y$ can be
covered by $N$ balls of the radius smaller than $(2a_2\rho_{t/N})^{-1}$
(where $a_2>0$ is the constant appearing in $f_{\mathrm{low}}$ in \eqref
{upper-es}), that is, the minimal $N$ for which
%
%e15 ###
\begin{equation}
\label{balls} \frac{\|x-y\|}{N}\leq\frac{1}{a_2\rho_{t/N}}.
\end{equation}
Observe that $q^*(r)\leq c_1 r^2$, $r\geq1$, implying $c_2
t^{-1/2}\leq\rho_t$ for all $t$ small enough. Hence, \eqref{balls}
holds with $N \geq\frac{(a_2 c_2 \|x-y\|)^2}{t}$.
Therefore, putting $y_0=x$ and $y_N=y$, we get
\begin{align*}
p_t(x,y)&= \int_\rn\dots\int
_\rn \Biggl(\prod_{i=1}^N
p_{t/N}(y_{i-1},y_i) \Biggr)dy_1
\dots dy_N
\\
&\geq\int_{B(y_0,(2a_2\rho_{t/N})^{-1})} \dots\int_{B(y_{N-1},
(2a_2\rho_{t/N})^{-1})}\prod
_{i=1}^N p_{t/N}(y_{i-1},y_i)
dy_i
\\
&\geq c_0 \rho_{t/N}^{Nn},
\end{align*}
where in the last line we used that
\[
p_{t/N} (y_{i-1},y_i)\geq2^{-1}a_1
\rho_{t/N}^n \quad \text{for all $y_i\in B
\bigl(y_{i-1},(2a_2\rho_{t/N})^{-1}
\bigr)$}.
\]
Thus, the transition probability density $p_t(x,y)$ is strictly positive.

Finally, for a signed Borel measure $\varpi$, define
%
%e16 ###
\begin{equation}
\label{h} h(r):=\sup_x |\varpi|\bigl\{ y: \, \|x-y\|\leq r\bigr\},
\end{equation}
where $|\varpi|:= \varpi^++\varpi^-$ is the total variation of
$\varpi$. Denote by $\hat{h}$ the Laplace transform of $h$.

\medskip

The following theorem is the main result of the paper. Let $t_0\in
(0,1]$ be small enough.

\begin{thm}\label{T-FK1}
Let $X$ be the Feller process constructed in Theorem~\emph{\ref{t-main}}.
Take a~signed Borel measure $\varpi$ such that its volume function
\eqref{h} satisfies
%
%e17 ###
\begin{equation}
\label{H1} \int_0^t\rho_s^{n+1}
\hat{h}(\rho_s)ds\leq C t^\zeta, \quad t\in[0,1],
\end{equation}
with some constants $C,\zeta>0$, where $\rho_t$ is given by \eqref
{rho1}. Then

\begin{itemize}

\item[\rm a)] There exists a continuous functional $A_t$ such that
\[
\Ee^x A_t=\int_0^t
\int_\rn p_s(x,y)\varpi(dy)ds;
\]

\item[\rm b)]
The semigroup $(T^A_t)_{t\geq0}$ is well defined, and its kernel
possesses a density $p^A_t(x,y)$ with respect to the Lebesgue measure
on $\rn$;

\item[\rm c)] There exist constants $a_i>0$, $1\leq i\leq4$, and a
family of sub-probability measures $\{\mathcal{R}_t, \, t\geq0\}$ such
that for
$t\in(0,t_0]$ and $x,y\in\rn$,
%
%e18 ###
\begin{equation}
\label{pa10} \rho_t^n f_{\mathrm{low}}\bigl((x-y)
\rho_t\bigr)\leq p^A_t(x,y)\leq
\rho_t^n \bigl(f_{\mathrm{up}}(\rho_t \cdot)
* \mathcal{R}_t \bigr) (y-x);
\end{equation}

here $f_{\mathrm{low}}$ and $f_{\mathrm{up}}$ are the function of the form \eqref{flu1}
with some constants $a_i$, $1\leq i\leq4$.
\end{itemize}
\end{thm}
\begin{remark}\label{bounds}
In general, $a_i$ in estimate \eqref{pa10} are  \emph{some}
constants, that may not coincide with those in estimate \eqref{upper-es}.
In order to simplify the notation, we assume that in Theorem~\ref
{t-main}, $a_1=a_3=1$, $a_2=a$, and $a_4=b$.
\end{remark}
Assumption \eqref{H1} can be relaxed, provided that more information
about the initial transition probability density is available.
Put
%
%e19 ###
\begin{equation}
\label{gtx} \mathfrak{g}_t(x):= \frac{1}{t^{\frac{n}{\alpha}}( 1+ \|x\|
/t^{1/\alpha
})^{d+\alpha}}, \quad t>0,
 \ x\in\rn.
\end{equation}
Note that for $d=n$, this function is equivalent to the transition
probability density of a symmetric $\alpha$-stable process in $\rn$
(that is, the process whose characteristic function is $e^{-t \|\xi\|
^\alpha}$).
Denote by $\mathcal{K}_{n,\alpha}$ the class of Borel signed measures
such that
%
%e20 ###
\begin{equation}
\label{kato10} \lim_{t\to0} \sup_x \int
_0^t \frac{|\varpi|\{ y: \, \|x-y\|\leq
s\}
}{s^{n+1-\alpha}}ds=0.
\end{equation}
The following lemma shows that for $d>n-\alpha$ the Kato class of
measures with respect to $\mathfrak{g}_t(x-y)$ coincides with
$\mathcal
{K}_{n,\alpha}$.
The proof uses the idea from~\cite{BJ07}, and will be given in Appendix \ref{appendix}.
\begin{lemma}\label{ka2}
A finite Borel signed measure $\varpi$ belongs to $S_K$ with respect to
$\mathfrak{g}_t(x-y)$, given by \eqref{gtx} with $d>n-\alpha$, if and
only if $|\varpi|\in\mathcal{K}_{n,\alpha}$.
\end{lemma}
\begin{cor}
In particular, it follows from Lemma~\emph{\ref{ka2}} that $\varpi\in S_K$
with respect to the transition probabiility density of a symmetric $\alpha$-stable
process if an only if $\varpi\in\mathcal{K}_{n,\alpha}$.
\end{cor}
In the proposition below, we state the ``compact'' upper bound for $p_t^A(x,y)$.
\begin{proposition}\label{pr1}
Let $X$ be a Feller process satisfying the conditions of
Proposition~\emph{\ref{p-kato1}}, and in addition assume that the transition
density $p_t(x,y)$ of $X$ is such that for all $t\in(0,1]$, $x,y\in
\rn$, the inequality
%
%e21 ###
\begin{equation}
\label{gam1} p_t(x,y)\leq c \mathfrak{g}_t(x-y), \quad
t\in(0,1], \ x,y\in\rn,
\end{equation}
where the function $\mathfrak{g}_t(x)$ is defined in \eqref{gtx} with
$d>n-\alpha$. Suppose that $\varpi\in\mathcal{K}_{n,\alpha}$. Then
%
%e22 ###
\begin{equation}
\label{pr-eq1} p^A_t(x,y)\leq C \mathfrak{g}_t(x-y),
\quad t\in(0,1], \ x,y\in \rn.
\end{equation}
\end{proposition}
\begin{remark}
{\rm a)} For  $X$ being a symmetric $\alpha$-stable-like
process, such a result is known, see {\rm\cite{So06}}. In particular,
the upper bound \eqref{pr-eq1} holds with~$n=d$. In our case, $X$ is
from a wider class; in particular, we do not assume the symmetry of the
initial process, and the method of constructing the Feynman--Kac
semigroup is completely different.

{\rm b)} In view of Lemma~\ref{ka2}, under the assumptions of this
proposition, we can take $\varpi\in\mathcal{K}_{n,\alpha}$ rather than
$\varpi\in S_K$ with respect to $\mathfrak{g}_t$, which is more
convenient for usage.
\end{remark}

In Section~\ref{exam}, we provide examples that illustrate
Theorem~\ref
{T-FK1} and Proposition~\ref{pr1}.

%s2.1 ###
\subsection{Discussion and overview}
\begin{enumerate}
\item \emph{On continuous additive functionals.}
Loosely speaking, there are two approaches for constructing continuous
additive functionals. One approach, which we described previously,
relies on the Dynkin theory of $W$-functionals.
Another approach, based on the Dirichlet form technique, establishes the
one-to-one correspondence between the class of positive continuous
additive functionals and the class of \emph{smooth measures}, see
\cite
[Lemmas~5.1.7,~5.1.8]{Fu80} or \cite[Thm.~5.1.4]{FOT94} in the case
when the process under consideration is symmetric; see also \cite
[Thm.~2.4]{MR92} for the non-symmetric case.
In this paper, we use Dynkin's approach as more appropriate in our
situation, in particular, we do not assume that the initial Markov
process $X$ is symmetric. Our standard reference in this paper is~\cite{Dy65}.

\item \emph{On the generator of $(T_t^A)_{t\geq0}$}. Suppose that the
Markov process $X$ and the \emph{positive} functional $A_t$ are as in
Proposition~\ref{p-kato1}. In this case, the semigroup $(T_t^A)_{t\geq
0}$ is contractive, and thus there exists a sub-Markov process with
transition sub-probability density $p^A_t(x,y)$. Formally, we can
describe the generator of $(T_t^A)_{t\geq0}$ as
%
%e23 ###
\begin{equation}
\label{la} \mathcal{L}^A = \mathcal{L}-\varpi,
\end{equation}
where $\mathcal{L}$ is the generator of the semigroup associated with
$X$, and $\varpi$ is the measure appearing in the characteristic of
$A_t$ (cf. \eqref{chit}), see \cite[Thms.~9.5, 9.6]{Dy65} for the
(equivalent) formulation. Nevertheless, in this framework the problem
of defining the domain $D(\mathcal{L}^A)$ of $\mathcal{L}^A$ still
remains open.
In the general case, that is, when $A$ can attain negative values, in
order to define the generator of (non-contractive) semigroup
$(T_t^A)_{t\geq0}$, we can use the quadratic form approach, see \cite
{AM91,AM92}, and also
\cite{CS03}.
\end{enumerate}

%s3 ###
\section{Proof of Theorem~\ref{T-FK1}}\label{Proo1}

%s3.1 ###
\subsection{Proof of statements a) and b)}

a) By the upper bound in \eqref{upper-es} on $p_t(x,y)$ (see also
Remark~\ref{bounds}), \eqref{H1} implies that $\varpi\in S_K$:
\begin{align*}
&\sup_{x\in\rn} \int_0^t \int
_\rn p_s(x,y)|\varpi|(dy)ds
\\
&\quad\leq \sup_{x\in\rn} \int_0^t
\int_\rn\int_\rn\rho_s^n
f_{\mathrm{up}}\bigl((y-x-z)\rho _s\bigr)Q_s(dz)|
\varpi|(dy) ds
\\
&\quad \leq b \sup_{x\in\rn} \int_0^t
\int_\rn\int_0^\infty
\rho_s^n |\varpi| \bigl\{ y: \,\|y-x-z\|\leq v/
\rho_s \bigr\} e^{-b v} dv Q_s(dz)ds
\\
&\quad \leq b \int_0^t \rho_s^{n+1}
\hat{h}(b \rho_s) ds\to0, \quad t\to0.
\end{align*}
Hence, applying \cite[Thm.~6.6]{Dy65}, we derive the existence of a
continuous functional $A_t$ with claimed characteristic.

\medskip

Statement b) is already contained in Proposition~\ref{p-kato1}.

%s3.2 ###
\subsection{Outline of the proof of c)}\label{FK-32}

For the proof of Theorem~\ref{T-FK1}(c), we use the Duhamel principle.
First, we show that the function $p^A_t(x,y)$ satisfies the integral equation
%
%e24 ###
\begin{equation}
\label{duh} p^A_t(x,y)= p_t(x,y)+\int
_0^t \int_\rn
p_{t-s}(x,z) p^A_s(z,y)\varpi(dz)ds,
\end{equation}
provided that the integral on the right-hand side converges.
We show that if the series
%
%e25 ###
\begin{equation}
\label{ser} \pi_t(x,y):= \sum_{k=1}^\infty
p_t^{\lozenge k } (x,y)
\end{equation}
converges, then it satisfies Eq.~\eqref{duh}. We derive an upper
estimate for the convolutions $p_t^{\lozenge k}(x,y)$, which guarantees
the absolute convergence of the series and allows to find the upper
estimate for $\pi_t(x,y)$.

Second, we show that on $(0,t_0]\times\rn\times\rn$ the solution
\eqref{ser} to \eqref{duh} is unique in the class of non-negative
functions $\{f(t,x,y)\geq0,\, t\in(0,t_0],\, x,y\in\rn\}$ such that
%
%e26 ###
\begin{equation}
\label{duh1} \int_\rn f(t,x,y)dy\leq C \quad\text{for
all} \ t\in(0,t_0], \, x\in\rn.
\end{equation}
We use the standard method, based on the Gronwall--Bellman inequality.

Finally, observe that the kernel $p^A_t(x,y)$ of $T^A_t$ belongs to the
class of functions satisfying \eqref{duh1}. Indeed, since for $A_t$ we
have \eqref{gaug1}, it follows that
%
%e27 ###
\begin{equation}
\label{duh2} \big| T_t^A f(x) \big| \leq c_1
\Ee^x e^{|A_t|}\leq c_2, \quad f\in
B_b\bigl(\rn\bigr), \ x\in\rn, \ t\in(0,t_0].
\end{equation}
Thus, $p_t^A(x,y)\equiv\pi_t(x,y)$ on $(0,t_0]\times\rn\times\rn$.

Before we prove that \eqref{ser} is the solution to Eq.~\eqref{duh} on
$(0,t_0]\times\rn\times\rn$, let us discuss a simple case when
$\varpi$ is the Lebesgue measure on $\rn$. In this case $h(r)=c_n
r^n$, and thus
assumption \eqref{H1} is satisfied:
\[
\int_0^t \rho_s^{n+1}
\hat{h}(\rho_s)ds= c_n t.
\]
Therefore, the procedure of estimation of convolutions reduces to those
treated in \cite[Lemmas 3.1, 3.2]{KK13a}.

Rewrite the upper bound in \eqref{upper-es} as
%
%e28 ###
\begin{equation}
\label{phi91} p_t(x,y)\leq C_1 t^{-1/2}
\bigl(g_t^{(1)} * Q_t \bigr) (y-x),
\end{equation}
where $C_1>0$ is some constant,
%
%e29 ###
\begin{equation}
\label{gt1} g_t^{(1)}(x):= t^{1/2}
g_t(x),
\end{equation}
and (cf. Remark~\ref{bounds})
%
%e30 ###
\begin{equation}
\label{gt} g_t(x):= \rho_t^n
f_{\mathrm{up}}(\rho_t x)=\rho_t^n
e^{-b \rho_t |x|}.
\end{equation}
This modification is technical, but proves to be useful for estimating
the convolutions $p^{\lozenge k }_t(x,y)$.
Let us estimate $p^{\lozenge k }_t(x,y)$.
Take now a sequence $(\theta_k)_{k\geq1} $ such that $0<\theta
_{k+1}<\theta_k$, $\theta_1=1$, and put
%
%e31 ###
\begin{equation}
\label{gtk} g_t^{(k)}(x):= t^{k/2}
g_t (\theta_k x), \quad k\geq1.
\end{equation}
Since $\rho_t$ is monotone decreasing, for $0<s<\frac{t}{2}$, we have
$\rho_{t-s}\leq\rho_{t/2}$. Note that $\rho_t\asymp\rho_{t/2}$; this
follows from condition \textrm{A1} and the definition of $\rho_t$; see
\cite{KK12a} for the detailed proof. Then, for $0<s<t/2$,
%
%e32 ###
\begin{align}
 \bigl(g^{(k-1)}_{t-s} *
g_s^{(1)} \bigr) (x)&\leq t^{k/2} \int
_\rn g_{t-s}( \theta_{k-1}x -
\theta_{k-1}y) g_s (\theta_{k-1}y)dy\notag
\\
& = t^{k/2} \theta_{k-1}^{-n} \int
_\rn g_{t-s} (\theta_{k-1} x -y)
g_s(y)dy\notag
\\
& \leq t^{k/2} \theta_{k-1}^{-n} \int
_\rn\rho_{t-s}^n \rho_s^n
e^{-\frac{b\rho_t \theta_k}{\theta_{k-1}} (|\theta
_{k-1}x-z|+|z|)-b\rho
_s (1-\frac{\theta_k}{\theta_{k-1}})|z|}dz\notag
\\
&\leq c_1 t^{k/2} \theta_{k-1}^{-n}
\rho_t^n e^{-b\rho_t \theta_k
|x|}\int_\rn
\rho_t e^{-b\rho_s (1-\frac{\theta_k}{\theta
_{k-1}})|z|}dz\notag
\\
&= D_k g_t^{(k)}(x),\label{gtk2}
\end{align}
where
$D_k
=c  (\theta_{k-1}-\theta_k )^{-n}$, $c=c_1 \int_\rn e^{-b|z|}dz$,
and in the second line from below, we used the triangle inequality and
monotonicity of $\rho_t$. In the case $t/2\leq s\leq t$, calculation
is similar.

By induction we can get
%
%e33 ###
\begin{equation}
\label{gkk} \big| p^{\lozenge k }_t(x,y) \big|\leq C_k
t^{\frac{k}{2}-1} \bigl( g_t^{(k)} * Q_t^{(k)}
\bigr) (y-x), \quad k\geq2,
\end{equation}
where
\[
C_k := c^{k-1}C_1^k
\frac{\varGamma^k(1/2)}{\varGamma(k/2)} \prod_{j=2}^k
\frac{1}{(\theta_{j-1}-\theta_j)^n},
\]
and for $k\geq2$
\[
Q_t^{(k)}(dw):=\frac{1}{B(\frac{k-1}{2},\frac{1}{2})} \int_0^1
\int_\real(1-r)^{(k-1)/2-1/2} r^{-1/2}
Q_{t(1-r)}^{(k-1)}(dw-u)Q_{tr}^{(1)}(du)dr.
\]
Since $\{Q_t^{(k)}, \, t>0, \, k\geq1\}$ is the sequence of
sub-probability measures and $g_t^{(k)}(x)\leq\rho_t^n t^{k/2}$, we obtain
\[
\big| p^{\lozenge k }_t(x,y)\big |\leq C_k t^{k-1}
\rho_t^n.
\]
Thus, to show the absolute convergence of the series $\sum_{k=1}^\infty
p^{\lozenge k }_t(x,y)$, we may check that $\sum_{k=1}^\infty
C_k<\infty
$. However, the behaviour of $C_k$ as $k\to\infty$ is rather
complicated. To see this, take, for example, $\theta_k =\frac{1}{2} +
\frac{1}{2k}$. Then
\[
C_k= c^{k-1} C_1^k
\frac{\varGamma^k(1/2)}{\varGamma(k/2)} \bigl(2^k k!(k-1)!\bigr)^n,
\]
and thus $C_k$ explodes as $k\to\infty$. Therefore, this procedure of
estimation of convolutions is too rough, and needs to be modified. For
this, we change the estimation procedure after some finite number of
steps; this allows us to control the decay of coefficients and, in such
a way, to prove that $\sum_{k=1}^\infty p^{\lozenge k }_t(x,y)<\infty$.

In the next subsection, we handle the general case, in particular,
\begin{itemize}
\item We give the generic calculation, which allows us to estimate the
convolution $ (g_{t-s}\,\sq\, g_s )(x)$;

\item We estimate the convolutions $p^{\lozenge k }_t(x,y)$, $k\geq2$;

\item We change the estimation procedure after $k_0$ steps, where $k_0$
is properly chosen, and estimate $p^{\lozenge(k_0+\ell) }_t(x,y)$,
$\ell\geq1$.
\end{itemize}

The change of the estimation procedure could be unnecessary if we would
know that $p_t(x,y)$ possesses a more regular upper bound than \eqref
{upper-es}. In this case, we obtain a sufficient control on the
coefficients $C_k$, $k\geq1$. This is exactly the case under the
conditions of Proposition~\ref{pr1}.

%s3.3 ###
\subsection{Representation lemma, generic calculation, and estimation
of convolutions}\label{aux}
\medskip

\begin{lemma}\label{pv}
The function $p_t^A(x,y)$ given by \eqref{pv10} satisfies Eq.~\eqref{duh}.
\end{lemma}
\begin{proof}
In the case when $X$ is a symmetric stable-like process and $\varpi
\in
S_K$ with respect to the transition probability density of $X$, the sketch of the
proof is given in~\cite{So06}. In the general case, the proof is the
same; in order to make the presentation self-contained, we present it
below. Using the equality
\[
e^{A_t}= \int_0^t
e^{A_t-A_s}dA_s+1,
\]
the strong Markov property of $X$, and the additivity of $A_t$ (cf.
\eqref{ad1}), we write
\begin{align*}
T^A_t f(x)&=\Ee^x \bigl[
f(X_t)e^{A_t} \bigr]
\\
&=\Ee^x f(X_t) + \Ee^x \Biggl[ \int
_0^t \bigl[ f(X_t) e^{A_t-A_s}
\bigr] dA_s \Biggr]
\\
&= \Ee^x f(X_t) + \Ee^x \Biggl[ \int
_0^t \Ee^{X_s} \bigl[
f(X_{t-s}) e^{A_{t-s}}\bigr] dA_s \Biggr]
\\
&=\Ee^x f(X_t)+ \Ee^x \int
_0^t T_{t-s}^A
f(X_s) dA_s. %&=T_t f(x)+ \int_0^t\int_\rn T_{t-s}^A f(y) p_s(x,y)\varpi(dy)ds,
\end{align*}
Observe that for $f\in B_b(\rn)$, we have
%
%e34 ###
\begin{equation}
\label{char10} \Ee^x \int_0^t
f(X_s)dA_s=\int_0^t
\int_\rn f(y)p_s(x,y)\varpi(dy)ds.
\end{equation}
Indeed, since $\chi_t= \chi_t^+-\chi_t^-$ with $\chi_t^\pm$ given by
\eqref{chipm} is the characteristic of $A_t$, Eq.~\eqref{char10} holds
for a finite linear combination of indicators. Approximating $f\in
B_b(\rn)$ by such linear combinations and passing to the limit, we get
\eqref{char10}.
\end{proof}

For $\theta\in[0,1]$, put
%
%e35 ###
\begin{equation}
\label{g10} g_{t,\theta}(x):= g_t(\theta x),
\end{equation}
where $g_t(x)$ is defined in \eqref{gt}, and
%
%e36 ###
\begin{equation}
\label{ph1} \phi_\nu(s):= \rho_s^{n+1}
\hat{h}(\nu\rho_s), \quad\nu>0,
\end{equation}
where $h$ is the volume function (cf. \eqref{h}) appearing in condition
\eqref{H1}. Lemma below gives the generic calculation, needed for the
proof of Theorem~\ref{T-FK1}.
\begin{lemma}\label{con-g}
For $\theta\in(0,1)$, we have
%
%e37 ###
\begin{equation}
\label{g-conv10} (g_{t-s}\,\sq\, g_s ) (x)\leq C\bigl[
\phi_{(1-\theta)b} (t-s)+\phi _{(1-\theta)b}(s)\bigr] g_{t,\theta}( x),
\quad x\in\rn, \ 0<s<t\leq1,
\end{equation}
where $C>0$ is some constant, independent of $\theta$, and $b>0$ comes
from the definition of $g_t$, see \eqref{gt}.
\end{lemma}
\begin{proof}
Take $\theta\in(0,1)$. Since by definition the function $\rho_t$ is
decreasing, we have
\[
\|x-z\|\rho_{t-s} +\|z\|\rho_s \geq\|x\|
\rho_t,
\]
which implies
\begin{align*}
&(g_{t-s} \,\sq\, g_s ) (x)
\\
&\quad\leq e^{-\theta b \|x-y\|\rho_t} \rho_{t-s}^n \rho_s^n
\int_\rn \bigl[ f_{\mathrm{up}}\bigl((z-x)
\rho_{t-s}\bigr) f_{\mathrm{up}}\bigl((y-z)\rho_s\bigr)
\bigr]^{(1-\theta)}|\varpi|(dz).
\end{align*}
By integration by parts we derive, using that $\rho_t$ is monotone decreasing, that
%
%e38 ###
\begin{align}
 &\int_\rn
\rho_{t-s}^n\rho_s^n
\bigl[f_{\mathrm{up}}\bigl((x-z) \rho_{t-s}\bigr) f_{\mathrm{up}}
\bigl((z-y)\rho _s\bigr)\bigr]^{1-\theta} |\varpi|(dz)\notag
\\
& \quad\leq\rho_{t/2}^n \int_\rn
\rho_s^n f_{\mathrm{up}}^{1-\theta} \bigl((z-y)
\rho _s\bigr)|\varpi|(dz)\notag
\\
& \quad\leq c_1 \rho_t^n \rho_s^n
\int_0^\infty|\varpi|\bigl\{ z: e^{-b(1-\theta
)\|z-y\|\rho_s }
\geq e^{-v}\bigr\} e^{-v}dv\notag
\\
&\quad= (1-\theta)b c_1 \rho_t^n
\rho_s^n \int_0^\infty|
\varpi|\bigl\{ z: \| z-y\|\leq v /\rho_s\bigr\} e^{-b(1-\theta)v}dv\notag
\\
&\quad\leq(1-\theta)b c_1 \rho_t^n
\rho_s^n \int_0^\infty
h(v /\rho_s) e^{- b(1-\theta)v}dv\notag
\\
&\quad= c_1 \rho_t^n \rho_s^{n+1}
\hat{h}\bigl(b(1-\theta)\rho_s\bigr)\notag
\\
&\quad= c_1 \rho_t^n \phi_{b(1-\theta)}(s).\label{conv10}
\end{align}
Similar estimate holds true for $s>\frac{t}{2}$, which finishes the proof of
\eqref{g-conv10}.
\end{proof}

Take a sequence $(\theta_k)_{k\geq1}$ such that
%
%e39 ###
\begin{equation}
\label{tet} \theta_1=1, \qquad\theta_k>0, \qquad
\theta_{k-1}>\theta_k, \quad k\geq2.
\end{equation}
Let
%
%e40 ###
\begin{equation}
\label{k0} k_0:= \biggl[\frac{n}{\alpha\zeta} \biggr],\vadjust{\goodbreak}
\end{equation}
where $\zeta$ is the parameter appearing in \eqref{H1}.
Define
%
%e41 ###
\begin{equation}
\label{kap} \kappa:= \min\bigl\{ b (\theta_{j-1}-
\theta_j), \ 1\leq j \leq k_0\bigr\},
\end{equation}
%
%e42 ###
\begin{equation}
\label{F1} F(t):= \int_0^t
\phi_\kappa(r)dr,
\end{equation}
and
%
%e43 ###
\begin{equation}
\label{gtk-new} \tilde{g}_t^{(k)}(x):= %
\begin{cases}
g_{t,\theta_k}(x) F^{k-1} (t), &\ 1\leq k\leq k_0, \\
e^{-b\theta_{k_0} \rho_t \|x\|} F^{k-k_0}(t), &\ k> k_0,
\end{cases}
\end{equation}
where $g_{t,\theta}(x)$ is defined in \eqref{g10}.

Finally, define inductively the sequence of measures
\[
\mathcal{R}_t^{(1)}(dw):= Q_t(dw) \quad
\text{if}\ k=1,
\]
%
%e44 ###
\begin{equation}
\label{rt} \mathcal{R}_t^{(k)}(dw):= \bigl(2 F(t)
\bigr)^{-1}\int_0^t\int
_\rn\bigl[ \phi_\kappa(t-s) +\phi _\kappa
(s) \bigr] Q_{t-s}(dw-u)\mathcal{R}_{s}^{(k-1)}(du)ds
\end{equation}
%
%Note that by construction $\mathcal{R}_t^{(k)}(\rn) \leq1$ for all $t
%\in(0,1]$.
if $k\geq2$. Since $(Q_t)_{t\geq0}$ is the family of sub-probability
measures (see Theorem~\ref{t-main}), we have
\begin{align*}
\mathcal{R}^{(2)}_t\bigl(\rn\bigr)\leq\bigl(2F(t)
\bigr)^{-1} \int_0^t \bigl[
\phi_\kappa (t-s)+\phi_\kappa(s)\bigr] Q_{t-s}\bigl(
\rn\bigr)Q_s\bigl(\rn\bigr) ds\leq1,
\end{align*}
and we can see by induction that $\mathcal{R}^{(k)}_t(\rn)\leq1$,
$t\in[0,1]$, for all $k\geq2$.

\begin{lemma}\label{Pk} For $k\geq2$ we have
%
%e45 ###
\begin{equation}
\label{P1-ek} \big|p_t^{\lozenge k} (x,y)\big|\leq\tilde{C}_k
\bigl(\tilde{g}_t^{(k)} * \mathcal{R}^{(k)}_t
\bigr) (y-x), \quad x,y\in\rn, \ t\in(0,1],
\end{equation}
where the sequence $ (\tilde{g}_t^{(k)} )_{k\geq1} $ is given by
\eqref{gtk-new}, $\mathcal{R}^{(k)}_t$ is defined in \eqref{rt},
$k\geq
2$, and for $k>k_0$, the constants $\tilde{C}_k$ can be expressed as
\[
\tilde{C}_k= C^{k-k_0} M,
\]
where $M,C>0$ are some constants.
\end{lemma}
\begin{proof}
We use induction. Rewrite the upper estimate on $p_t(x,y)$ in the form
\eqref{phi91}. For $k=2$ we get, using \eqref{phi91} and \eqref
{conv10}, the following estimates:
%
%e46 ###
\begin{align}
\big|p_t^{\lozenge2}(x,y)\big|&
\leq C_1^2\int_0^t\int
_{\mathbb{R}^{2n}} \biggl[\int_{\rn}
\tilde{g}_s^{(1)}(z-x-w_1)
\tilde{g}_{t-s}^{(1)}(y-z-w_2) |\varpi|(dz) \biggr]\notag
\\
& \quad\cdot Q_{t-s}(dw_1) Q_s(dw_2)ds\notag
\\
&\leq C_2 \int_\rn g_{t,\theta_2} (x-w)
\Biggl\{ \int_0^t \bigl[\phi
_{b(\theta
_1-\theta_2)}(t-s)+\phi_{b(\theta_1-\theta_2)}(s)\bigr]\notag
\\
&\quad\cdot\int_\rn Q_{t-s}(dw-u)Q_s(du)ds
\Biggr\}\notag
\\
&\leq C_2 \int_\rn g_{t,\theta_2} (x-w)
\Biggl\{ \int_0^t \bigl[\phi
_{\kappa
}(t-s)+\phi_{\kappa}(s)\bigr]\notag
\\
&\quad\cdot\int_\rn Q_{t-s}(dw-u)Q_s(du)ds
\Biggr\}\notag
\\
&\leq2 C_2 F(t) \bigl(g_{t,\theta_2} * \mathcal{R}^{(2)}
\bigr) (y-x)\notag
\\
&= 2 C_2 \bigl(\tilde{g}_t^{(2)} *
\mathcal{R}^{(2)} \bigr) (y-x),\label{P10}
\end{align}
where $C_1>0$ comes from \eqref{phi91}, and in the third line from
below we used that by the definition of $\kappa$ and monotonicity of
$\phi_\nu$ in $\nu$,
\[
\phi_{b(\theta_{j-1}-\theta_j)}(t)\leq\varphi_\kappa(t), \quad t\in(0,1].
\]
Suppose that \eqref{P1-ek} holds for some $2\leq k\leq k_0$. Then
%
%e47 ###
\begin{align}
  \big|p_t^{\lozenge(k+1)} (x,y)\big|
&\leq2^{k-1}C_k C_1 \int_0^t
\int_\rn \bigl(\tilde{g}_{t-s}^{(1)} *
Q_{t-s} \bigr) (z-x)\notag
\\
&\quad\cdot \bigl(\tilde{g}_s^{(k)} *
\mathcal{R}^{(k)}_s \bigr) (y-z)dzds\notag
\\
&=2^{k-1}C_k C_1 \int_0^t
\int_\rn\int_\rn \bigl(
\tilde{g}_{t-s}^{(1)} \sq\tilde{g}_{s}^{(k)}
\bigr) (y-x-w_1-w_2)\notag
\\
&\quad\cdot Q_{t-s} (dw_1) \mathcal{R}^{(k)}_s(dw_2)ds.\label{e10}
\end{align}
By the same argument as those used in the proof of Lemma~\ref{con-g},
we have
\begin{align*}
\bigl(\tilde{g}_{t-s}^{(1)} \,\sq\,\tilde{g}_{s}^{(k)}
\bigr) (x) & \leq ( g_{t-s,\theta_k}\,\sq\, g_{s,\theta_k} ) (x)F^{k-1}
(t)
\\
&\leq c_{k+1} g_{t,\theta_{k+1}}(x) F^{k-1} (t) \bigl[
\phi_{b(\theta
_{k-1}-\theta_k)}(t-s)+\phi_{b(\theta_{k-1}-\theta_k)}(s)\bigr]
\\
&=c_{k+1} \bigl(F(t)\bigr)^{-1}\bigl[\phi_\kappa(t-s)+
\phi_\kappa(s)\bigr] \tilde {g}_{t}^{(k+1)}(x).
\end{align*}
Substituting this estimate into \eqref{e10}, performing the change of
variables and normalizing, we get \eqref{P1-ek} for $2\leq k\leq k_0$.

Take $c_0>0$. Note that for some $c_1>0$, we have $c_0 \rho_t \leq
\rho
_{c_1 t}$, $t\in(0,1]$. Then, by \eqref{H1},
\begin{align*}
\int_0^t \rho_t^{n+1}
\hat{h}(c_0 \rho_t)dt \leq c_2 \int
_0^t \rho_{c_1
t}^{n+1}
\hat{h}(\rho_{c_1 t})dt \leq c_3 \int_0^{c_1t}
\rho_t^{n+1} \hat{h}(\rho_t)dt\leq
c_4 t^\zeta.
\end{align*}
Therefore, taking $k_0$ as in \eqref{k0}, we get
%
%e48 ###
\begin{equation}
\label{k00} \rho_t^n F^{k_0}(t)\leq
c_5 t^{-n/\alpha+k_0\zeta} \leq c_6, \quad t\in[0,1].
\end{equation}
In such a way, on the $(k_0+1)$-th step, we obtain
\begin{align*}
\bigl(\tilde{g}_{t-s}^{(k_0)} \,\sq\,\tilde{g}_{s}^{(1)}
\bigr) (x) &\leq c e^{-b\theta_{k_0}\rho_t \|x\|}\int_\rn
e^{-b\rho_s (1-\theta
_{k_0})\|
z-x\|}|\varpi|(dz)
\\
&= c e^{-b\theta_{k_0}\rho_t \|x\|} \int_0^\infty|\varpi|\bigl
\{ z:\, \rho _s b(1-\theta_{k_0}) \|z-x\|\leq r \bigr\}
e^{-r}dr
\\
& \leq c e^{-b\theta_{k_0}\rho_t \|x\|} \phi_{b(1-\theta
_{k_0})}(s)
\\
&\leq c \tilde{g}_t^{(k_0+1)} (x) \phi_{\kappa}(s)
F^{-1}(t)
\end{align*}
(cf. \eqref{conv10}), where in the last line we used the inequality
$\kappa<b(1-\theta_{k_0})$ and the monotonicity of $\phi_\nu$ in
$\nu$.
Using this estimate, we derive
%
%e49 ###
\begin{align}
  p^{\lozenge(k_0+1)}_t(x,y)&\leq
C_{k_0} C_1 \int_0^t
\int_\rn\int_\rn \bigl(
\tilde{g}_{t-s}^{(k_0)} \sq\tilde {g}_{s}^{(1)}
\bigr) (y-x-w_1-w_2)\notag
\\
&\quad\cdot Q_{s} (dw_1) \mathcal{R}^{(k_0)}_{t-s}(dw_2)ds\notag
\\
&\leq2cC_1 C_{k_0} \cdot \bigl( \tilde{g}_t^{(k_0+1)}
* \mathcal {R}^{(k_0+1)}_t \bigr) (y-x).\label{ptk0}
\end{align}
%
%$M= 2^{k_0-1}C_{k_0} $.
Then \eqref{P1-ek} follows by induction. Indeed, assume that \eqref
{P1-ek} holds for $k=k_0+\ell-1$. For  $\ell\geq2$ we get
\[
\bigl(\tilde{g}_{t-s}^{(k_0+\ell-1)}\,\sq\,\tilde{g}_s^{(1)}
\bigr) (x)\leq cF^{\ell-1}(t) e^{-b\theta_{k_0}\rho_t \|x\|} \phi_\kappa(s)=c
F^{-1}(t) \tilde {g}_t^{(k_0+\ell)}(x)
\phi_\kappa(s).
\]
Therefore,
%e50 ###
\begin{align}
 \big|p_t^{\lozenge(k_0+\ell)} (x,y)\big|&
\leq (2C_1 c )^{\ell-1} C_1 C_{k_0} \int
_0^t \int_\rn \bigl(
\tilde{g}_{t-s}^{(k_0+\ell-1)} * \mathcal {R}^{(k_0+\ell-1)}_{t-s}
\bigr) (z-x)\notag
\\
&\quad\cdot \bigl(\tilde{g}_s^{(1)}
*Q_{s} \bigr) (y-z)dzds\notag
\\
&= C_{k_0}(2C_1 c)^{\ell} \bigl(
\tilde{g}_t^{(k_0+\ell)}* \mathcal {R}_t^{(k_0+\ell)}
\bigr) (y-x). \label{e10}
\end{align}
\end{proof}
\begin{remark}
As we observed in the proof, the estimation procedure depends on
condition \textrm{H1}, which guarantees the existence of the number
$k_0$ such that \eqref{k00} holds. In general, without \textrm{H1} we
cannot guarantee the existence of such a number, which is crucial in
our approach. For example, suppose that $\rho_s\asymp s^{-1}$ for small
$s$, and take the measure $\varpi$ such that
\[
h(r)\asymp\frac{1}{\ln^2 r}, \quad r\in(0,1].
\]
By the Tauberian theorem, we have
$\hat{h}(\lambda) \asymp [\lambda\ln^2\lambda ]^{-1} $ for
large $\lambda$.
Therefore,
$\phi_\nu(t) \sim|\ln t|^{-1}$ as $t\to0$, and thus the integral
$F(t)$ diverges.
Nevertheless, assumption \textrm{H1} can be dropped, if the function
$p_t(x,y)$ possesses a more precise upper bound. We discuss this
question later in Section~\ref{Proo2}.
\end{remark}

%s3.4 ###
\subsection{Proof of statement c)}

From \eqref{e10} we get for all $x,y\in\rn$,
%
%e51 ###
\begin{equation}
\label{pl} \big|p_t^{\lozenge(k_0+\ell)} (x,y)\big|\leq M \bigl(C F(t)
\bigr)^\ell, \quad \ell\geq1,
\end{equation}
where $M= C_{k_0}$ and $C=2C_1 c$. Without loss of generality, assume
that $C\geq1$. Since $F(t)\to0$ as $t\to0$, there exists $t_0>0$,
such that
%
%e52 ###
\begin{equation}
\label{CF} C F(t)<1/2, \quad t\in(0,t_0].
\end{equation}
Thus, for $t\in(0,t_0]$, the series \eqref{ser} converges absolutely
and is the solution to \eqref{duh}.\goodbreak

Let us show that the integral equation \eqref{duh} possesses a unique
solution in the class of functions $\{f(t,x,y)\geq0, \, t\in
(0,t_0],\, x,y\in\rn\}$, such that
%
%e53 ###
\begin{equation}
\label{duh11} \int_\rn f(t,x,y)dy\leq c, \quad t
\in(0,t_0], \,x\in\rn.
\end{equation}
Then the series \eqref{ser} is a unique representation of the
Feynman--Kac kernel $p_t^A(x,y)$ for $t\in(0,t_0]$, $x,y\in\rn$.

Suppose that there are two solutions $p_t^{(1),A}(x,y)$ and
$p_t^{(2),A}(x,y)$ to \eqref{duh}.
Put $\tilde{p}_t^A (x,y):= |p_t^{(1),A}(x,y) -p_t^{(2),A}(x,y)|$ and
$v_t(x):= \int_\rn\tilde{p}_t^A (x,y)dy$.
Then, by \eqref{duh} we have
%
%e54 ###
\begin{equation}
\label{ser4} v_t(x)\leq\int_0^t
\int_\rn p_{t-s}(x,z) v_s(z)
\varpi(dz)ds.%\leq
%\int_0^t \Big[ \int_\rn p_{t-s}(x,z) \varpi(dz)\Big] \sup_y v_s(y)ds.
\end{equation}
By induction we get
%
%e55 ###
\begin{equation}
\label{ser4} v_t(x)\leq\int_0^t
\int_\rn p_{t-s}^{\lozenge k}(x,z)
v_s(z)\varpi(dz)ds.
\end{equation}
Note that there exists $c>0$ such that $p^{\lozenge(k_0+1)}_t(x,y)\leq
c$ for all $t\in(0,t_0]$, $x,y\in\rn$ (cf. \eqref{ptk0}). In such a
way, by the finiteness of  measure $\varpi$, we get
%
%e56 ###
\begin{equation}
\label{ser5} v_t(x)\leq c_1 \int_0^t
\int_\rn v_s(z)\varpi(dz)ds\leq
c_2 \int_0^t \tilde{v}_sds,
\end{equation}
where $\tilde{v}_s:= \sup_{z\in\rn}v_s(z)$. Taking $\sup_{x\in\rn}$
in the left-hand side of \eqref{ser5}, we derive
%
%e57 ###
\begin{equation}
\label{ser6} \tilde{v}_t \leq c_2 \int
_0^t \tilde{v}_s ds, \quad t
\in(0,t_0].
\end{equation}
Applying the Gronwall--Bellman lemma, we derive $\tilde{v}_t\equiv0$
for all $t\in(0,t_0]$. Thus, the solution to \eqref{duh} is unique in
the class of functions
\[
\bigl\{f(t,x,y)\geq0, \, t\in(0,t_0],\, x,y\in\rn\bigr\}
\]
satisfying \eqref{duh11}.

\medskip

Estimating series \eqref{ser} from above, we get an upper bound in
\eqref{pa10} with $f_{\mathrm{up}}$ of the form \eqref{flu1} and
\[
\mathcal{R}_t(dw)= c_0 \sum
_{k\geq1} c^k \mathcal{R}_t^{(k)}(dw),
\]
with some $c\in(0,1)$ and the normalizing constant $c_0>0$ chosen so
that $\mathcal{R}_t(\rn)\leq1$ for all $t\in(0,t_0]$.

For the lower bound, observe that by \eqref{gtk-new} we have
%
%e58 ###
\begin{equation}
\label{pk} \big|p_t^{\lozenge k} (x,y)\big|\leq C(k_0)
\rho_t^n F(t), \quad2\leq k\leq k_0.
\end{equation}
By \eqref{pl} and \eqref{CF} we get
\[
\sum_{\ell\geq1} p_t^{\lozenge(k_0+\ell)}(x,y)
\leq2 M C F(t),\quad t\in(0,t_0],
\]
which, together with \eqref{pk} and the observation that $\rho_t$ is
decreasing, yields the estimate
%
%e59 ###
\begin{equation}
\label{p0psi} \bigg| \sum_{k=2}^\infty
p_t^{\lozenge k} (x,y) \bigg|\leq C_0 F(t) \rho
_t^n, \quad t\in(0, t_0],
\end{equation}
where $C_0>0$ is some constant.
Therefore, choosing $t_0$ small enough, we have by the lower bound in
\eqref{upper-es} the inequalities
%
%e60 ###
\begin{align}
 p^A_t(x,y) &\geq\rho_t^n
f_{\mathrm{low}}\bigl((y-x)\rho_t\bigr)- C_0F(t)
\rho_t^n\notag\\
&\geq c \rho_t^n
f_{\mathrm{low}}\bigl((y-x)\rho_t\bigr), \quad t
\in(0,t_0].\label{pa-low}
\end{align}
\qed

%s4 ###
\section{Proof of Proposition~\ref{pr1}}\label{Proo2}

Since the proof of the proposition follows with minor changes from the
proof of the upper estimate in \cite[Thm.~3.3]{So06}, we only sketch
the argument.
For $(t,x,y)\in(0,t_0]\times\rn\times\rn$, put
\[
I_0(t,x,y):= \mathfrak{g}_t(x-y),\quad I_k
(t,x,y)= \int_0^t \int_\rn
\mathfrak{g}_{t-s}(x-z) I_{k-1}(s,z,y)\varpi(dz)ds.
\]
By the same argument as in \cite{So06}, we can get
\[
\big|I_k(t,x,y) \big|\leq c^k \mathfrak{g}_t(y-x),
\quad k\geq1,\ t\in(0,t_0],
\]
where $c\in(0,1)$ is some constant. Thus, for $k\geq1$, we have
%
%e61 ###
\begin{equation}
\label{Pkg} \big|p_t^{\lozenge k} (x,y) \big|\leq c^k
\mathfrak{g}_t(x-y),\quad x,y\in\rn, \ t\in(0,t_0].
\end{equation}
This proves the convergence of the series \eqref{ser} and the upper
estimate \eqref{pr-eq1}.\qed

\begin{remark}
Let us briefly discuss the crucial difference between the proofs of
Theorem~\ref{T-FK1} and Proposition~\ref{pr1}. We changed the procedure
of estimation of $p_t^{\lozenge k}(x,y)$ after a certain step, which
was possible due to \eqref{H1}.
In the case when we have a single-kernel estimate for $p_t(x,y)$, for
example, \eqref{gam1}, we can drop condition \eqref{H1}. In fact, it is
enough to require that $\varpi\in S_K$ with respect to $\mathfrak
{g}_t(y-x)$. This happens because in the case of the single-kernel
estimate of type \eqref{gam1}, it is possible to show that the
convolutions $p_t^{\lozenge k} (x,y)$ satisfy the upper bound \eqref
{Pkg} with $c\in(0,1)$, which implies the convergence of the series
\eqref{ser}.
\end{remark}

%s5 ###
\section{Examples}\label{exam}

As one might observe, the scope of applicability of Theorem~\ref{T-FK1}
heavily relies on the properties of the initial process $X$. To assure
the existence of such a process, we applied Theorem~\ref{t-main}. Below we give the  examplesin which  condition \textrm{A1} is satisfied. Since
conditions \textrm{A2}--\textrm{A4} are easy to check, we may assume
that the functions $a(x)$ and $m(x,u)$ are appropriate. We confine
ourselves to the case when the measure $\mu$ in the generator of $X$
is ``discretized $\alpha$-stable; up to the author's knowledge, in this
case the corresponding Feynman--Kac semigroup was not investigated.
Examples  below  illustrate that our approach is applicable also in the situation when the  ``L\'evy-type measure'' $m(x,u)\mu(du)$ related to the
initial process $X$ is not absolutely continuous with respect to the
Lebesgue measure.

%and Proposition~\ref{pr1}
%The case when $X$ is a $\alpha$-stable stable-like process
%
\begin{example}
a) Consider a ``discretized version'' of an $\alpha$-stable L\'evy
measure in $\rn$.
Let $m_{k,\upsilon}(dy)$ be the uniform distribution on a sphere
$\mathbb{S}_{k,\upsilon}$ centered at~$0$ with radius $2^{-k\upsilon}$,
$\upsilon>0$, $k\in\mathbb{Z}$.
Consider the L\'evy measure
%
%e62 ###
\begin{equation}
\label{mu1} \mu(dy)=\sum_{k=-\infty}^\infty2^{k\gamma}m_{k,v}(dy),
\end{equation}
where $0<\gamma<2\upsilon$. In {\rm\cite{K13}}, it is shown that for
such a L\'evy measure condition \textrm{A1} is satisfied, and
%
%e63 ###
\begin{equation}
\label{rho-al} \rho_t\asymp t^{-1/\alpha}, \quad t\in(0,1],
\end{equation}
where $\alpha=\gamma/\upsilon$.

Take some functions $a(\cdot): \rn\to\real$ and a non-negative bounded
function $m(\cdot,\cdot)$ defined on $ \rn\times\rn$ satisfying
assumptions \textrm{A2}--\textrm{A4}. By Theorem~\ref{t-main} the
operator of the form \eqref{lxd0} with $\mu$, $a(x)$, and $m(x,u)$ as
before can be extended to the generator of a Feller process $X$ that
admits the transition density $p_t(x,y)$ satisfying \eqref{upper-es}.

Let $\varpi$ be a finite Borel measure, and let $h$ be its volume
function, see \eqref{h}. Let us show that if the inequality
%
%e64 ###
\begin{equation}
\label{H2} \int_0^t \frac{h(v)}{v^{n+1-\alpha}}dv
\leq c_1 t^\zeta,\quad t\in(0,1],
\end{equation}
for some $\zeta>0$, then we have \eqref{H1}. Using \eqref{rho-al},
changing variables, and applying the Fubini theorem, we derive
\begin{align*}
\int_0^t \rho_s^{n+1}
\hat{h}(\rho_s)ds&\leq \int_0^t
s^{-\frac{n+1}{\alpha}} \hat{h}\bigl(c_2s^{-1/\alpha}\bigr) ds
\\
%\alpha\int_0^{t^{1/\alpha}} \int_0^\infty\frac{h(\tau u)}{\tau^{n+1-
%\alpha}}e^{-c u}dud\tau\\
&= \alpha\int_0^\infty
\Biggl[ \int_0^{t^{1/\alpha} v} \frac
{h(\tau
)}{\tau^{n+1-\alpha}}d\tau
\Biggr] v^{n-\alpha} e^{-c_2v}dv.
\end{align*}
Denote by $I(t)$ the right-hand side in this expression. Applying
\eqref
{H2}, we get
\begin{align*}
I(t)\leq c_1 \int_0^\infty
\bigl(t^{1/\alpha} v\bigr)^\zeta v^{n-\alpha} e^{-c_2v}dv
\leq c_3 t^{\zeta/\alpha}.
\end{align*}
In particular, if $h(v)\leq c v^d$, $d>n-\alpha$, then \eqref{H2} holds.

Thus, by Theorem~\ref{T-FK1}, the Feynman--Kac semigroup
$(T_t^A)_{t\geq0}$ is well defined, and the kernel $p^A_t(x,y)$
satisfies \eqref{pa10} with some constants $a_i$, $1\leq i\leq4$, and
some family of sub-probability measures $(\mathcal{R}^{(k)})_{t\geq0}$.

b) Consider now the one-dimensional situation. In this case, the L\'evy
measure $\mu$ from \eqref{mu1} is just
%
%e65 ###
\begin{equation}
\label{mu2} \mu(dy)=\sum_{n=-\infty}^\infty2^{n\gamma}
\bigl(\delta _{2^{-n\upsilon
}}(dy)+\delta_{-2^{-n\upsilon}}(dy) \bigr).
\end{equation}
Let $X$ be a L\'evy process with characteristic exponent
\[
\psi(\xi):= \int_\rn\bigl(1-\cos(\xi u)\bigr)\mu(du).
\]
In {\rm\cite{KK12a}} we show that if $1<\alpha<2$, then the transition probability
density $p_t(x,y)$ of $X$, $X_0=x$, is continuous in $(t,x,y)\in
(0,\infty)\times\real\times\real$ and admits the following upper bound:
%
%e66 ###
\begin{equation}
\label{cantor1} p_t(x,y) \leq c t^{-1/\alpha} \bigl(1+
|y-x|/t^{1/\alpha} \bigr)^{-\alpha}, \quad t\in(0,1], \ x,y\in
\real.
\end{equation}
Note that the right-hand side of \eqref{cantor1} is of the form \eqref
{gtx} with $d=0$. Thus, the conditions of Proposition~\ref{pr1} are
satisfied, and we can construct the Feynman--Kac semigroup for the
related functional $A_t$ and the transition density $p_t(x,y)$, and get
the upper bound for the function $p^A_t(x,y)$ with $\rho_t\asymp
t^{-1/\alpha}$, $t\in(0,1]$.

To end this example, we remark that it is still possible to construct
the upper bound for such $p_t(x,y)$ for $\alpha\in(0,1)$ of the form
$t^{-n/\alpha} f(xt^{-1/\alpha})$, but the function $f$ in this upper
bound might not be integrable; see {\rm\cite{KK12a}} for details. Note
that the upper bound \eqref{cantor1} is non-integrable in $\rn$ for
$n\geq2 $.
\end{example}

\begin{example}
Consider the L\'evy measure
%
%e67 ###
\begin{equation}
\label{e:mn} \nu_0(A) = \int_{\mathbb{R}^n} \int
_0^\infty\I_{A}(rv) r^{-1-\alpha}
\,dr \mu_0(dv)\,,\quad\alpha\in(0,2),
\end{equation}
where $\alpha\in(0,2)$, $\mu_0$ is a finite symmetric non-degenerate
$($that is, not concentrated on a linear subspace of $\rn)$ measure on
the unit sphere $\mathbb{S}^n$ in $\rn$. Suppose that
there exists $d>0$ such that for small $r$ we have
\[
\nu_0\bigl(B(x,r)\bigr)\leq C r^d, \quad\|x\|=1.
\]
For $d+\alpha>n$, it is shown in {\rm\cite{BS07}} that the
corresponding L\'evy process $X$, $X_0=x$, admits the transition probability
density $p_t(x,y)$, which satisfies
%
%e68 ###
\begin{equation}
\label{cantor20} p_t(x,y)\leq c t^{-n/\alpha} \bigl( 1+\|y-x\|
t^{-1/\alpha} \bigr)^{-d-\alpha}, \quad t>0, \ x,y\in\rn.
\end{equation}
In the forthcoming paper
{\rm\cite{BKS14}}, we construct a class of L\'evy-type processes that
admit the transition densities bounded from above by the left-hand side
of \eqref{cantor20}.
Thus, taking $\varpi\in\mathcal{K}_{n,\alpha}$, we may apply
Proposition~\ref{pr1}.
\end{example}

\section*{Acknowledgments} The author thanks Alexei Kulik and Niels
Jacob for valuable discussions and comments, and the anonymous referee
for helpful remarks and suggestions. The DFG Grant Schi~419/8-1 and the
Scholarship of the President of Ukraine for young scientists
(2012--2014) are gratefully acknowledged.

\appendix

%s6 ###
\section{Appendix}\label{appendix}

\begin{proof}[Proof of Lemma~\ref{ka2}]
We follow the idea of the proof of \cite[Lemma~11]{BJ07}.
Without loss of generality, assume that $\varpi$ is non-negative.
Suppose first that $\varpi\in\mathcal{K}_{n,\alpha}$. Using
integration by parts and the Fubini theorem, we get
%
%e69 ###
\begin{align}
 &\int_0^t
\int_\rn\mathfrak{g}_s (x-y) \varpi(dy)ds\notag\\
&\quad\asymp\int_0^t \int_\rn
s^{-n/\alpha} \bigl( 1\wedge\bigl(s^{1/\alpha
}/\| x-y\|\bigr)
\bigr)^{\alpha+d}\varpi(dy)ds\notag
\\
&\quad= \int_0^t s^{-n/\alpha} \varpi\bigl\{ y:
\, \|x-y\|\leq s^{1/\alpha}\bigr\} ds\notag
\\
&\qquad+ \int_0^t s^{-n/\alpha} \int
_{\|x-y\|>s^{1/\alpha}} \biggl( \frac
{s^{1/\alpha}}{\|x-y\|} \biggr)^{d+\alpha}
\varpi(dy)ds\notag
\\
&\quad=\alpha \biggl( 1+ \frac{d+\alpha}{d+2\alpha-n} \biggr) \int_0^{t^{1/\alpha
}}
\frac{\varpi\{y:\, \|x-y\|\leq v\}}{v^{n+1-\alpha}}dv\notag
\\
&\qquad+ \frac{\alpha(d+\alpha)}{d+2\alpha-n} t^{\frac{d+2\alpha
-n}{\alpha}} \int_{t^{1/\alpha}}^\infty
\frac{\varpi\{ y:\, \|x-y\|
<v\} }{v^{d+1+\alpha}}dv.\label{ka-eq1}
\end{align}
Since $\varpi\in\mathcal{K}_{n,\alpha}$, the first term tends to 0 as
$t\to0$. Further, since $d>n-\alpha$ and the measure $\varpi$ is
finite, we have
\[
t^{\frac{d+2\alpha-n}{\alpha}} \sup_x \int_{1}^\infty
\frac
{\varpi\{
y:\, \|x-y\|<v\} }{v^{d+1+\alpha}}dv\to0, \quad t\to0.
\]
Let us show that
%
%e70 ###
\begin{equation}
\label{low-lim} \sup_x t^{\frac{d+2\alpha-n}{\alpha}} \int
_{t^{1/\alpha}}^1 \frac
{\varpi\{ y:\, \|x-y\|<v\} }{v^{d+1+\alpha}}dv.
\end{equation}
Let $K_0\equiv K_0(t):= [t^{-1/\alpha}]+1$; note that
$K_0(t)t^{1/\alpha
}\to1$ as $t\to0$. We have
\begin{align*}
&t^{\frac{d-\epsilon+2\alpha-n}{\alpha}} \int_{t^{1/\alpha}}^1
\frac
{\varpi\{ y:\, \|x-y\|<v\} }{v^{d+1+\alpha}}dv
\\
&\quad\leq\sum_{k=1}^{K_0} \biggl(
\frac{1}{k} \biggr)^{(d-n+2\alpha)/\alpha} \int_{kt^{1/\alpha}}^{(k+1)t^{1/\alpha}}
\frac{\varpi\{ y:\, \|
x-y\|
<v\} }{v^{n+1-\alpha}}dv.
\end{align*}
Since $d>n-\alpha$, we have
$\sum_{k=1}^\infty k^{-(d-n+2\alpha)/\alpha}<\infty$.
Since $\varpi\in\mathcal{K}_{n,\alpha}$, we have
\begin{align*}
\max_{1\leq k\leq K_0(t)} & \sup_x\int
_{kt^{1/\alpha
}}^{(k+1)t^{1/\alpha}} \frac{\varpi\{ y:\, \|x-y\|<v\}
}{v^{n+1-\alpha
}}dv\longrightarrow0,
\quad t\to0.
\end{align*}
Thus, we arrive at \eqref{low-lim}.
This proves that \eqref{kato10} implies that $\varpi\in S_K$ with
respect to $\mathfrak{g}_t(y-x)$.

The converse is straightforward.
\end{proof}

\end{document}